\documentclass[onefignum,onetabnum]{siamonline250211}

\usepackage{lipsum}
\usepackage{amsfonts}
\usepackage{graphicx}
\usepackage{epstopdf}
\usepackage{algorithmic}
\ifpdf
  \DeclareGraphicsExtensions{.eps,.pdf,.png,.jpg}
\else
  \DeclareGraphicsExtensions{.eps}
\fi

\usepackage{enumitem}
\setlist[enumerate]{leftmargin=.5in}
\setlist[itemize]{leftmargin=.5in}

\newsiamremark{remark}{Remark}
\newsiamremark{hypothesis}{Hypothesis}
\crefname{hypothesis}{Hypothesis}{Hypotheses}
\newsiamthm{claim}{Claim}
\newsiamremark{fact}{Fact}
\crefname{fact}{Fact}{Facts}
\newsiamremark{assu}{Assumption}
\newsiamremark{prop}{Proposition}
\newsiamremark{thm}{Theorem}

\def\mP{{\mathcal P}}

\def\mE{{\mathcal E}}
\def\mX{{\mathcal X}}

\def\mY{{\mathcal Y}}

\def\mI{{\mathcal I}}

\def\mH{{\mathcal H}}

\def\R{{\mathbb R}}

\def\N{{\mathbb N}}

\def\NI{{\mathrm{NI}}}

\def\eps{{\varepsilon}}

 \def\1{{\mathbf 1}}

\usepackage{amsmath,stmaryrd,bbm,hyperref,geometry,color}
\usepackage{amssymb}
\usepackage{verbatim}

 \newcommand{\po}{\left(}
\newcommand{\pf}{\right)}
\newcommand{\co}{\left[}
\newcommand{\cf}{\right]}
\newcommand{\cco}{\llbracket}
\newcommand{\ccf}{\rrbracket}

\newcommand{\T}{\mathbb T} 

\newcommand{\dd}{\text{d}}
\newcommand{\na}{\nabla}

\headers{Time-Averaged Gradient Descent Dynamics for Zero-Sum Games}{Y. Lu and P. Monmach\'e}

\title{Convergence of Time-Averaged Mean Field Gradient Descent  Dynamics for Continuous Multi-Player Zero-Sum Games\thanks{Submitted to the editors DATE.
}}

\author{Yulong Lu\thanks{School of Mathematics, University of Minnesota, Minneapolis, MN 55455, USA
  (\email{yulonglu@umn.edu}).}
\and Pierre Monmarch{\'e}\thanks{LJLL and LCT, Sorbonne Universit\'e, 4 place Jussieu, 75005 Paris France
  (\email{pierre.monmarche@sorbonne-universite.fr}).}}

\usepackage{amsopn}

\ifpdf
\hypersetup{
  pdftitle={An Example Article},
  pdfauthor={D. Doe, P. T. Frank, and J. E. Smith}
}
\fi

\begin{document}

\maketitle

\begin{abstract}
The approximation of mixed Nash equilibria (MNE) for zero-sum games with mean-field interacting players has recently raised much interest in machine learning. In this paper we propose a mean-field gradient descent dynamics for finding the MNE of zero-sum games involving $K$ players with $K\geq 2$.  The evolution of the players' strategy distributions follows coupled mean-field gradient descent flows with momentum, incorporating an exponentially discounted time-averaging of gradients.  First, in the case of a fixed entropic regularization, we prove an exponential convergence rate for the mean-field dynamics to the mixed Nash equilibrium with respect to the total variation metric. This improves a previous polynomial convergence rate for a similar time-averaged dynamics with different averaging factors. Moreover, unlike previous two-scale approaches for finding the MNE, our  approach treats all player types on the same time scale. We also show that with a suitable choice of decreasing temperature, a simulated annealing version of the mean-field dynamics converges to an MNE of the initial unregularized problem.
\end{abstract}

\begin{keywords}
 Zero-sum games, Mixed Nash equilibrium, Mean-field, Gradient descent ascent%
\end{keywords}

\begin{MSCcodes}
35Q89, 49N80, 91A16, 90C47
\end{MSCcodes}

\section{Introduction}
Zero-sum minmax games are ubiquitous in machine learning. Typical examples include Generative Adversarial Networks (GANs) (\cite{goodfellow2020generative}), adversarial training (\cite{madry2018towards}) and reinforcement learning (
\cite{busoniu2008comprehensive,bucsoniu2010multi,zhang2021multi}). In this paper, 
we consider games between $K$   players with $K\geq 2$. The loss of a player  $i\in\cco 1,K\ccf$  playing strategy $x_i \in\mathcal{X}_i := \mathbb{T}^{d_i}$  against all of other players, denoted by $x_{-i} := (x_j)_{j\neq i}$,  is denoted by $U_i(x) = U_i(x_i, x_{-i}) $. Here we use $\mathbb{T}^{d_i}$ to denote the $d_i$-dimensional torus. We remark that the torus assumption, while somewhat restrictive or impractical in game theory, are indeed prevalent in recent literature on the theoretical analysis of optimization dynamics in the mean-field settings \cite{ma2021provably,domingo2020mean,lu2023two}. Although we expect majority of the results in the paper can be extended to unbounded strategy spaces when an appropriate confining potential is incorporated into the loss function, we restrict our analysis on the torus to avoid the additional technicalities of dealing with the unboundedness of the state space. 

We write $\mX=\prod_{i=1}^K\mX_i$. Throughout the paper we restrict our attention on sum of  pairwise zero-sum games, in the following sense:  there exist $W_{ij}:\mX_i\times \mX_j \mapsto \R$ with $W_{ij}=-W_{ji}$ such that
$$
    U_i(x) = \sum_{j\neq i} W_{ij}(x_i,x_j)\,.
$$
Further assumptions on $W_{ij}$ can be found in Assumption \ref{ass-multi}.   
The finite dimensional game problem seeks a (pure) Nash equilibrium $(x_1^{\ast}, \cdots, x_K^{\ast})$ such that 
\begin{equation} \label{eq:Ui}
U_i (x_i^{\ast},  x_{-i}^{\ast}) \leq U_i (x_i,  x_{-i}^{\ast}), \quad \forall i\in  \cco 1,K\ccf.    
\end{equation}
Note that in the definition of pure Nash equilibrium  we assume that the $i$-th player is aiming to minimize a loss $U_i$ (instead of maximizing a pay-off, which is perhaps more commonly used in the game literature). When $\mX_i$ are discrete sets, the resulting games reduces to the so-called {\em polymatrix games}  \cite{bregman1987methods,bergman1998separable}; see more discussions in Section \ref{sec:relatedwork}.  Algorithmic approximation of the pure Nash equilibria of losses $\{U_i\}_{i=1}^K$ is a very challenging problem when $U_i$ lacks convexity. In fact, in that case the pure Nash equilibria may not exist (\cite{Nash1951}). To avoid the ill-posed issue of the solutions, there has been a growing interest in finding the {\em mixed Nash equilibria} (\cite{glicksberg1952further}). 
To define the mixed Nash equilibrium (MNE), let $\nu^{i} \in \mathcal{P}(\mX_i)$ be a probability distribution of the mixed strategy used by the $i$-th player for   $i \in \cco 1,K\ccf$. The averaged loss of the $i$-th player against the distributions of the mixed strategies is given by 
\begin{equation}\label{eq:mne}
    \mE_i(\nu^{i}, \nu^{-i}) = \int_{\mX_{-i}} U_i(x_i, x_{-i}) \nu^{-i} (dx_{-i}),
\end{equation}
where $\mX_{-i}  = \prod_{j\neq i}^K\mX_j$ and $\nu^{-i} := \{\nu^{j}\}_{j\neq i}$. 
We say a set of measures $(\nu^{1,\ast}, \cdots, \nu^{K,\ast})$ is an MNE of the multi-player game if 
$$
\mE_i(\nu^{i,\ast}, \nu^{-i,\ast}) \leq \mE_i(\nu^{i}, \nu^{-i,\ast}), \forall  \nu^i \in \mP(\mX_i),  \forall i\in  \cco 1,K\ccf.
$$
 While MNE exists in much generality, algorithmically finding it is still highly nontrivial. Motivated by the recent line of works for two-player zero-sum games, we consider the approximation of an MNE of the game objectives $\{\mE_i\}_{i=1}^K$ with the MNE of their entropy-regularization. More precisely, we consider the game equipped with the entropy-regularized  losses: \begin{equation}\label{eq:Etau}
     \mE_{i,\tau}(\nu^i, \nu^{-i}) := \mE_{i}(\nu^i, \nu^{-i}) + \tau \mH(\nu^i), \forall i\in  \cco 1,K\ccf,
 \end{equation}
where $\mH(\nu) = \int \log \frac{d\nu}{dx} d\nu $ is the entropy of the probability measure $\nu$ and  $\tau>0$ is a finite regularization parameter (or temperature).
 Notice that for every $\tau>0$ and every fixed $\nu^{-i}$,  the functional $\mu^i \mapsto  \mathcal E_{i,\tau}(\mu^i,\nu^{-i})$ is strongly convex over $\mu^i \in \mathcal P (\mX_i)$. Moreover,  its unique minimizer is given by  
$$\rho_{U_i\star \nu^{-i}} \propto \exp(-\tau^{-1} U_i \star \nu^{-i}).
$$ 
Here we adopt the notation  $U_i \star \nu^{-i}(x_i) := \int U_i(x_i, x_{-i}) \nu^{-i}(dx^{-i})$. 
This allows us to  obtain that $\nu^*_{\tau}=(\nu^{1,\ast}_{\tau},\dots,\nu^{K,\ast}_{\tau})$ is a MNE $\{\mE_{i,\tau}\}_{i\in  \cco 1,K\ccf}$ if and only if it solves the fixed point equation
\begin{equation}\label{eq:fixedpointmulti}
    \forall i\in\cco 1,K\ccf,\qquad \nu^{i,\ast}_{\tau} = \rho_{U_i\star \nu^{-i,\ast}_{\tau}} \propto \exp \po - \tau^{-1} U_i \star \nu^{-i,\ast}_{\tau}\pf\,.
\end{equation}
Furthermore, the solution to the fixed point equation \eqref{eq:fixedpointmulti} is unique (see Proposition \ref{prop:NEmulti}). 

\paragraph{Two-player zero-sum games} When $K=2$, the games discussed above reduce to well-studied two-player zero-sum games ($K=2$). In this case,  
a two-player zero-sum game can be formulated as a minmax optimization problem of finding the Nash equilibria of a game objective function $W: \mX_1 \times \mX_2 \mapsto \R$, i.e. $\min_{x\in \mX_1}\max_{y\in \mX_2} W(x,y)$. Moreover, a  pair of probability distributions $(\nu^\ast,\mu^\ast)$ is an MNE if  \begin{equation}
\begin{aligned}
   &  \mE(\nu^\ast,\mu)\leq  \mE(\nu^\ast,\mu^\ast) \leq \mE(\nu,\mu^\ast), \quad \forall (\nu,\mu)\in  \mP(\mX_1) \times \mP(\mX_2),\\
    & \text{ where }  \mE(\mu,\nu):= \int_{\mX_1} \int_{\mX_2} W(x,y) \mu(dx)\nu(dy).
   \end{aligned}
\end{equation}
Similarly, one can define  the entropy-regularized minmax problem 
$$
\min_{\nu\in \mP(\mX_1)} \max_{\mu\in \mP(\mX_2)} \mE_{\tau}(\nu,\mu), \text{ where } \mE_{\tau}(\nu,\mu):= \mE(\nu,\mu) + \tau \mH(\nu) - \tau \mH(\mu).
$$
The unique MNE $(\nu^\ast, \mu^\ast)$ of $\mE_\tau$ is given by the unique solution to the fixed point equation (see e.g. \cite{domingo2020mean}) 
\begin{equation}\label{eq:fpt}\begin{aligned}
    \nu^\ast(dx)  \propto \exp\Big( - \tau^{-1} W \star \mu^\ast (y) \Big), \quad 
 \mu^\ast(dy)   \propto \exp\Big( \tau^{-1} W \star \nu^\ast (x))\Big),
\end{aligned}\end{equation}
where we have used the notations 
$$W \star \mu(x) = \int_{\mX_2}W(x,y) \mu(dy) \text{ and } W \star \nu(y) = \int_{\mX_1}W(x,y) \nu(dx). $$ 

A natural approach to approximate the MNE defined by \eqref{eq:fpt} is to consider the following Wasserstein gradient descent ascent flow of the energy $\mE_{\tau}$: 
\begin{equation}
\label{eq:WGDA}
\begin{aligned}
\partial_t \nu_t & = \nabla \cdot \Big(\nu_t \nabla_x W\star \mu_t \Big) +\tau \Delta \nu_t,\\
\partial_t \mu_t & = -\nabla \cdot \Big(\mu_t \nabla_y W\star \nu_t\Big) + \tau\Delta \mu_t.
\end{aligned}
\end{equation}
The dynamics \eqref{eq:WGDA} can be interpreted as the mean field limit ($N \rightarrow \infty$) of the algorithmically implementable interacting Langevin dynamics 
\begin{equation}\label{eq:gdaparticle}\begin{aligned}
    dX^i_t & = -\frac{1}{N} \sum_{j=1}^N \nabla_x W(X^i_t, Y^j_t) dt + \sqrt{2\tau} dW^i_t, \\
dY^i_t & = \frac{1}{N} \sum_{j=1}^N \nabla_y W(X^j_t, Y^i_t) dt + \sqrt{2\tau} dB^i_t.
\end{aligned}\end{equation}
For this reason, the dynamics  \eqref{eq:WGDA} is often referred to as as the {\em mean-field gradient descent ascent (mean-field GDA)}. It can also be viewed as the minimax analog of the mean-field Langevin dynamics studied in \cite{nitanda2022convex, chizat2022mean, hu2021mean}, which naturally arise in the mean-field analysis of the optimization of two-layer neural networks. Despite its simplicity and popularity, the convergence of the mean-field GDA \eqref{eq:WGDA} remains an open problem; for a detailed discussion on this problem, see \cite{wang2024open}. A two-scaled variant of \eqref{eq:WGDA} was previously investigated by \cite{lu2023two}, where the right-hand side of the second equation in \eqref{eq:WGDA} is multiplied by a small constant $\eta \ll 1$ (and  by \cite{ma2021provably} for the case of an infinitesimal $\eta$).  However, the slow progression of the slower dynamics in this two-scale approach significantly limits its efficiency and practical applicability.

\paragraph{Time-averaged mean-field gradient descent dynamics.} An interesting  time-averaged variant of \eqref{eq:WGDA} was recently introduced in  \cite{kim2023symmetric}, where the gradient velocity $\nabla_x \int_{\mY} W(x,y) d\mu_t(y)$ (and similarly the other gradient) in \eqref{eq:WGDA} is replaced by a weighted time-average $$\Big(\int_{0}^t \beta_s ds\Big)^{-1}\int_{0}^t \beta_s \nabla_x \int_{\mY} W(x,y) d\mu_s(y) ds $$  along its history with a discounting weight $\beta_t = t^{r}$ with $r>-1$. With this weight,  \cite{kim2023symmetric} managed to prove that for a fixed positive temperature $\tau>0$, the time-averaged dynamics with the two-players running  the same time-scale converges with an algebraic rate in the Nikaid\^o-Isoda (NI) error (\cite{nikaido1955note}). This result presents a promising alternative approach for designing uniformly scaled convergent dynamics for computing the MNE in two-player games.

In this paper, we study a multi-player analog of the time-averaged mean-field dynamics proposed by \cite{kim2023symmetric}. More precisely, we consider
\begin{equation}
\label{flow-multi-type}
\forall i\in\cco 1,K\ccf,\quad \left\{
\begin{array}{rcl}
\partial_t \nu_t^i   &= &  \na_{x_i} \cdot \po \nu_t^i \na_{x_i} U_i \star \nu_t^{-i} \pf + \tau \Delta_{x_i} \nu_t^i   \\  
\partial_t \hat\nu_t^i &= & \alpha  \po  \nu_t^i  - \hat \nu_t^i\pf,
\end{array}\right.
\end{equation}
where $\alpha>0$ and  $\nu_t^i$ is  the distribution of the mixed strategies used by the $i$-th  player at time $t$. 
Note that  the second  equation of \eqref{flow-multi-type} is explicitly solvable as
$$
\hat\nu_t^i =   e^{-\alpha t}\hat\nu_0^i +  \int_0^t \alpha  e^{-\alpha (t-s)}  \nu_s^i \dd s.
$$
Therefore the first  equation of \eqref{flow-multi-type} indicates that the distribution of $i$-th player interacts with the historical distributions of mixed strategies of all other players, but discounted by an exponential factor $e^{-\alpha t}$. The constant $\alpha>0$ controls the decaying rate of the discounting weight. We also note that given any initial condition $\nu_0,\hat\nu_0 \in \mP(\mX)$, the existence of uniqueness of the global solution to \eqref{flow-multi-type-annealed} follow directly from the classical  well-posedness of the theory of nonlinear Mckean-Vlasov-type PDEs (\cite{funaki1984certain,sznitman1991topics}, see also \cite[Proposition 1.1]{refId0} with time-averaging).
Observe that any stationary solution of \eqref{flow-multi-type} satisfies $\nu_t = \hat\nu_t=\nu^{\ast}_{\tau}$ with $\nu^{\ast}_{\tau}$ solving \eqref{eq:fixedpointmulti}. 

In the two-player setting ($K=2$), the dynamics  \eqref{flow-multi-type} reduces to 
\begin{equation}
\label{flow}
\left\{
\begin{array}{rcl}
\partial_t \nu_t   &= &  \na_x \cdot \po \nu_t \na_x W\star \hat\mu_t \pf +  \tau \Delta_x \nu_t   \\
\partial_t \mu_t   &= &   - \na_y \cdot \po \mu_t  \na_y W\star \hat\nu_t \pf + \tau \Delta_y \mu_t  \\   
\partial_t \hat\nu_t &= & \alpha  \po \nu_t - \hat\nu_t \pf\\
\partial_t \hat\mu_t &= & \alpha  \po  \mu_t - \hat\mu_t \pf. 
\end{array}\right. 
\end{equation}

We remark that our dynamics \eqref{flow} coincides with the time-averaged dynamics studied in \cite{kim2023symmetric} with a discounting weight  $\beta_t = \alpha e^{\alpha t}\mathbf{1}[0,t] + \delta_0$. It was  shown in \cite{kim2023symmetric} that the time-averaged dynamics with a polynomial discounting factor leads to an algebraic convergence rate. This motivates us to study whether a different choice of the weight, such as the exponential discounting factor, can lead to a faster rate of convergence. Specifically, we are interested in the following question:

{\em \textbf{Question 1:} Does the mean-field gradient descent dynamics \eqref{flow-multi-type} converges to the unique MNE of the energies $\{\mE_{i,\tau}\}_{i=1}^K$ with an exponential rate? }

As the first main result of the paper, we shall provide an affirmative answer to \textbf{Question 1}. Furthermore, inspired by the simulated annealing results for the mean-field Langevin dynamics (\cite{chizat2022mean,nitanda2022convex})  and a two-scale mean-field GDA dynamics (\cite{lu2023two}), we would like to investigate the convergence of the  time-averaged mean-field dynamics in the annealed regime where $\tau_t \rightarrow 0$ as $t \rightarrow \infty$. 

{\em \textbf{Question 2:} Does the annealed time-averaged mean-field gradient descent dynamics \eqref{flow-multi-type}  with a decreasing temperature $\tau = \tau_t$ converge to an MNE of the unregularized energies $\{\mE_i\}_{i=1}^K$ defined in \eqref{eq:mne}?}

\subsection{Summary of contributions}
We highlight the major contributions of the present paper as follows.

\begin{itemize}

\item We present a time-averaged mean field gradient descent dynamics \eqref{flow-multi-type} for approximating the MNE of a class of multi-player game problem, generalizing a similar dynamics proposed by \cite{kim2023symmetric} for two-player games. 

    \item We prove that for a fixed positive temperature $\tau$ and a sufficiently small $\alpha>0$, the time-averaged mean-field dynamics \eqref{flow-multi-type} converges exponentially to the MNE of the regularized game objectives  $\{\mE_{i,\tau}\}_{i=1}^K$.  We also quantify the dependence of the convergence rate explicitly in terms of $\alpha, \tau$ and bounds on the game losses $\{U_{i}\}_{i=1}^K$; see Theorem \ref{thm: bdnustat} for the precise statement.

    \item We show that by adopting a cooling schedule $\tau_t$ that decay with a logarithimic rate and a polynomial decreasing $\alpha_t$, the time-averaged mean-field  dynamics \eqref{flow-multi-type} converges to an MNE of the unregularized game objectives $\{\mE_i\}_{i=1}^K$ with respect to the Nikaid\` o-Isoda error (defined in \eqref{eq:NIerror-multi2}); see Theorem \ref{thm:mainanneal} for the precise statement.

\end{itemize}

\subsection{Related works}\label{sec:relatedwork}

\paragraph{Computational complexity of zero-sum  polymatrix games} Computing the mixed Nash equilibria (even approximately) of general games can be computationally intractable in the sense that approximate Nash equilibria is not achievable within polynomial time \cite{daskalakis2009complexity,chen2009settling}. When the space of strategies is discrete, the game problem defined in  \eqref{eq:Ui} belongs to a subclass of game problems called {\em polymatrix game} \cite{bregman1987methods,bergman1998separable}, which plays an essential role in showing the PPAD-hardness of games \cite{cai2011minmax,daskalakis2009complexity,chen2009settling,daskalakis2013complexity,etessami2010complexity}. In particular, in the discrete stategy space setting, our problem \eqref{eq:Ui} reduces to a pairwise zero-sum polymatrix game for which it has been shown \cite{daskalakis2009network,cai2011minmax} that a Nash equilibrium can be computed in polynomial-time using linear programming, despite that general polymatrix games are PPAD-complete. Nevertheless, the study of algorithmic guarantees of multi-player games in the continuous state space setting is still largely open.  

\paragraph{Continuous two-player games} For two-player games, the problem reduces to a minmax problem which received considerable attention recently. The most natural methods for finding the pure Nash equilibria are the gradient descent ascent methods for which many theoretical guarantees have been obtained under various assumptions on the game objectives. Many works either study the convergence of GDA in the convex-concave setting \cite{daskalakis2019last} in which pure equilibria exist  or its convergence to local equilibria in the non-convex settings \cite{daskalakis2018limit,lin2020gradient,yang2020global,yang2022faster,doan2022convergence}. Different from these works, the present paper studies the global convergence guarantees of algorithms for finding the mixed Nash equilibria of continuous multi-player games without convexity assumptions.

\paragraph{Mean field analysis}
The mean field perspective provides a valuable insight to the study of convergence analysis of dynamics for nonconvex optimization. Specifically, the mean-field Langevin dynamics plays a crucial role in the convergence convergence of gradient descent algorithms for neural networks, including  two-layer networks (\cite{mei2018mean,sirignano2020mean,rotskoff2022trainability,chizat2018global}),  residual networks (\cite{lu2020mean,wojtowytsch2020banach}) and general deep networks (\cite{araujo2019mean,sirignano2022mean,nguyen2019mean}). Global convergence of the mean-field Langevin dynamics has been studied in \cite{mei2018mean,chizat2018global,chizat2022mean,nitanda2022convex,hu2021mean}. Local convergence rates in non-convex cases have been established in \cite{MonmarcheReygner}.

In the context of min-max optimization, the mean-field GDA dynamics was first derived by  \cite{domingo2020mean,hsieh2019finding} from the interacting particle system \eqref{eq:gdaparticle}. The works \cite{lu2023two} and \cite{ma2021provably} proved the global convergence of the mean-field GDA dynamics to the mixed Nash equilibrium of entropy-regularized game objective in the two-scale setting where the ascent dynamics is running with a much faster or slower speed compared to the descent dynamics; see also \cite{an2025convergence}. Without the scale-separation, the convergence of the mean-field GDA dynamics remains an open question (\cite{wang2024open}), except  when the pay-off function $W(x,y)$ is  strongly-convex in $x$ and strongly-concave in $y$ \cite{cai2024convergence}. The work \cite{ding2024papal} studied a particle-based primal-dual algorithm for finding the mixed Nash equilibra.  For the time-averaged mean-field GDA dynamics, \cite{kim2023symmetric} obtained an algebraic convergence rate. This work improves such a convergence rate to an exponential rate by using a different discounting factor  in the time-average of the velocity. 

We remark that an alternative gradient flow with respect to the Wasserstein-Fisher-Rao (WFR) metric (\cite{chizat2018interpolating,kondratyev2016new,laschos2019geometric}) has also been used in the analysis of min-max optimization. \cite{domingo2020mean}  adopts a WFR-gradient flow for two-player zero-mean continuous games and proves that the time-averaging of the WFR-dynamics converges to the MNE in the regime where Fisher-Rao dynamics dominates the Wasserstein counterpart. In \cite{wang2022exponentially}, the authors introduced a  particle algorithm based on the implicit time-discretization of WFR-gradient flow for finding the MNE in the space of atomic measures. The authors proved a local exponential convergence of the proposed algorithm under certain non-degeneracy assumptions on the game objective and the  assumption that the MNE is unique. Lastly, the recent work (\cite{lascufisher,lascu2025entropic}) proved an exponential convergence for the pure Fisher-Rao gradient flow in the setting of the min-max games. 

Our main results parallel those obtained in \cite{lu2023two}, with the notable extension of results for zero-sum games involving $K$ players ($K\geq 2$). However, the technical approaches used in \cite{lu2023two} and in our work differ significantly. The exponential convergence demonstrated in \cite{lu2023two} heavily relies on the scale-separation assumption, enabling the construction of a specific Lyapunov function to measure the exponential convergence of the dynamics. In contrast, in this paper, the players are treated at the same time-scale, although the coupling between their evolution, governed by the parameter $\alpha$, is required to be small enough. This  makes straightforward the extension to more than $2$ player, with an algebraic dependency in $K$. At a technical level, we leverage careful entropy–entropy production estimates along the dynamics. In particular, the zero-sum assumption of the game introduces important cancellations of ``bad" terms in the entropy production, allowing it to be effectively controlled by the entropy itself. This control enables us to close the entropy–entropy production argument.

\subsection{General notations} 
We use $\mP(\mX)$ to denote the space of probability measures on the state space $\mX$.  Given two probability measures $\mu_1$ and $\mu_2$,  the relative entropy (or Kullback-Leibler divergence) $\mH(\mu_1|\mu_2)$ and the relative Fisher information $\mI(\mu_1|\mu_2)$ between two probability measures $\mu_1$ and $\mu_2$ are defined by 
$$\begin{aligned}
    \mH(\mu_1|\mu_2)  = \int \log\Big(\frac{d\mu_1}{d\mu_2}\Big) d\mu_1,\ 
    \mI(\mu_1|\mu_2)  = \int \Big|\nabla \log \Big(\frac{d\mu_1}{d\mu_2}\Big) \Big|^2 d\mu_1
\end{aligned}
$$
if $\mu_1$ is absolutely continuous with respect to $\mu_2$ and $+\infty$ otherwise. For $p\geq 1$, the $L^p$ Wasserstein distance $\mathcal W_p(\mu_1,\mu_2)$ between $\mu_1$ and $\mu_2$ is defined by 
$$
W_p(\mu_1,\mu_2) = \inf_{\gamma \in \Gamma(\mu_1,\mu_2)} \Big(\int |x-y|^p \gamma(dx dy)\Big)^{1/p},
$$
where $\Gamma(\mu_1,\mu_2)$ is the set of couplings between $\mu_1$ and $\mu_2$.  
We use $\|\mu_1-\mu_2\|_{TV}$ to denote the total variation distance between $\mu_1$ and $\mu_2$. For $K \in \N$,  we use $\cco 1,K\ccf$ to denote the set of integers $\{1,\cdots, K\}$. 

\section{Assumptions and Main Results}

\subsection{Settings and notations} \label{sec:setting}

Recall that we consider games between $K$   players with $K\geq 2$. Assume that the loss of a player  $i\in\cco 1,K\ccf$  playing strategy $x_i \in\mathcal{X}_i $ against all of other players $x_{-i} = (x_j)_{j\neq i}$,  is given by $U_i(x) = U_i(x_i, x_{-i}) $. We make the following assumption on the sum of pairwise zero-sum games:
\begin{assu}\label{ass-multi}
For all distinct pair $i,j\in\cco 1,K\ccf$,  there exist $W_{ij}:\mX_i\times \mX_j \mapsto \R$ with $W_{ij}=-W_{ji}$ such that
\begin{equation}
    \label{pairwise}
    U_i(x) = \sum_{j\neq i} W_{ij}(x_i,x_j)\,.
\end{equation}
Moreover, assume that $W_{ij}\in C^2(\mX_i\times \mX_j)$ for all distinct $i,j\in\cco 1,K\ccf$. Denote constants $M_0,M_1, M_2, L$  by  
\[M_0=\max_{i\in\cco1,K\ccf} \| U_{i}\|_\infty\,, L=\max_{i\in\cco1,K\ccf} \| \nabla U_{i}\|_\infty,  M_1 = \sum_{i=1}^K   \| U_i\|_\infty, \, M_2 =  \max_{i\in\cco 1,K\ccf}\sum_{j\neq i}   \|\na_{xy}^2 W_{i,j}\|_\infty \,. \]  
\end{assu}
Recall the game objectives $\{\mE_i\}_{i=1}^K$ and $\{\mE_{i,\tau}\}_{i=1}^K$ defined in and \eqref{eq:mne} and \eqref{eq:Etau} respectively. Recall also that the MNE $\nu^*_{\tau}=(\nu^{1,\ast}_{\tau},\dots,\nu^{K,\ast}_{\tau})$ associated to $\{\mE_{i,\tau}\}_{i=1}^K$ solves the fixed point equation~\eqref{eq:fixedpointmulti}. 
To measure the discrepancy of a collection of measures $\nu=(\nu^1,\dots,\nu^K)$ to the MNE of the entropy-regularized energies $\{\mE_{i,\tau}\}_{i=1}^K$, we define  the Nikaid\^o-Isoda (NI) error (\cite{nikaido1955note}) of $\nu$ associated to $\{\mE_{i,\tau}\}_{i=1}^K$ by 
\begin{equation}\label{eq:NIerror-multi}
 \NI_{\tau}(\nu)  := \sum_{i=1}^K \co  \mathcal E_{i,\tau}(\nu^i,\nu^{-i}) - \inf_{\mu^i\in\mathcal P(\mX_i)}   \mathcal E_{i,\tau}(\mu^i,\nu^{-i})\cf = \tau \sum_{i=1}^K \mathcal H\po \nu_i |\rho_{U_i\star\nu^{-i}}\pf .
\end{equation}
Similarly, we also define the NI error of $\nu$ associated to the unregularized energies $\{\mE_{i}\}_{i=1}^K$ as  
\begin{equation}\label{eq:NIerror-multi2}
 \NI(\nu)  := \sum_{i=1}^K \co  \mathcal E_{i}(\nu^i,\nu^{-i}) - \inf_{\mu^i\in\mathcal P(\mX_i)}   \mathcal E_{i}(\mu^i,\nu^{-i})\cf.
\end{equation}
To close the setting, we present a  proposition which establishes that there exists a unique equilibrium measure $\nu^\ast_\tau$ satisfying \eqref{eq:fixedpointmulti}. 

\begin{prop}\label{prop:NEmulti} 
Under Assumption~\ref{ass-multi}, for any $\tau>0$, there exists a unique Nash equilibrium $\nu^\ast_\tau = (\nu^{1,\ast}_{\tau},\dots,\nu^{K,\ast}_{\tau})$ associated to the energies $\{\mE_{i,\tau}\}_{i=1}^K$. Moreover, for any probability measures  $\nu = (\nu^1,\dots,\nu^K)$, it holds that 
    \begin{equation}
        \label{eq:entropy_ineq}
        \NI_\tau (\nu) \geqslant \tau \sum_{i=1}^K \mathcal H\po \nu^i|\nu^{i,\ast}_{\tau}\pf 
    \end{equation}
\end{prop}
We refer the  proof of Proposition \ref{prop:NEmulti}  to Section \ref{sec:proof1}.

\subsection{Convergence time-averaged mean-field gradient descent dynamics at a fixed temperature}  \label{sec:fixedtemp} 

 Recall the time-averaged mean-field gradient descent dynamics  \eqref{flow-multi-type} whose solution is given by $\nu_t=\bigotimes_{i=1}^K \nu_t^i$ and $\hat \nu_t = \bigotimes_{i=1}^K  \hat \nu_t^i$. 
Given these measures, we also define the proximal-Gibbs measure $\rho_t=\bigotimes_{i=1}^K \rho_t^i$ with
\[\rho_t^i \propto \exp \po -\tau^{-1} U_i \star \hat\nu_t^{-i}\pf,\, i\in\cco 1,K\ccf. \]
With the definitions above, it is straightforward to verify that 
\begin{equation}
    \label{eq:NIhat}
\NI_{\tau}(\hat\nu_t) = \tau\mathcal H(\hat\nu_t|\rho_t).
\end{equation}
To track the convergence of \eqref{flow-multi-type} to the equilibrium, we define the following Lyapunov function 
\[s_t := \mathcal H(\nu_t|\rho_t) + \mathcal H(\hat\nu_t|\rho_t) = \sum_{i=1}^K \mathcal H(\nu^i_t|\rho^i_t) + \mathcal H(\hat\nu^i_t|\rho^i_t) \,. \]
The proposition below establishes the exponential contraction of $s_t$ along \eqref{flow-multi-type}. 
 
\begin{prop}\label{prop:Lyapunov-multi}
Under Assumption~\ref{ass-multi}, set
\begin{equation}
\kappa =\frac{\exp(M_0/\tau)}{8\pi^2 }\,,\qquad \lambda = \min\po \frac{\tau}{2\kappa},\frac{\alpha}{4}\pf \,,\qquad \bar \alpha_0= \frac{\tau}{53 \kappa \max(1,(\kappa M_2/\tau)^2) }
 \,. \label{eq:kappa}
 \end{equation}
 Assume that $\nu_0^i,\hat\nu_0^i,i\in  \cco 1,K\ccf$ have finite entropies and that $\alpha \in(0,\bar \alpha_0]$. Then for all $t\geqslant 0$, along the flow~\eqref{flow-multi-type},
\[ s_t \leqslant e^{-\lambda   t}s_0  \,.\]

\end{prop}
Proposition \ref{prop:Lyapunov-multi} shows that both the measure $\nu_t$ and the time-averaged measure $\hat \nu_t$ are getting closer to the proximal Gibbs measures $\rho_t$ with respect to the relative entropy as $t$ tends to infinity. 
The proof of Proposition \ref{prop:Lyapunov-multi} can be found in Section \ref{sec:prooffor23}. This  proposition is a key ingredient to obtain the convergence of \eqref{flow-multi-type} to its equilibrium as presented in the theorem below.

\begin{thm}\label{thm: bdnustat}
     Under the assumptions and with the notations of Proposition \ref{prop:Lyapunov-multi}, we have  for all $t\geq 0$ that 
    \[ \sum_{i=1}^K   \mathcal H \po \hat \nu_t^i|\nu^{i,\ast}_\tau \pf  \leqslant  e^{-\lambda t}  s_0 \]
    and that 
    \[\sum_{i=1}^K \|\nu_t^i - \nu^{i,\ast}_\tau \|_{TV}^2 \leqslant 12  e^{-\lambda t}  s_0\,. \]
\end{thm}

\begin{remark}
Theorem \ref{thm: bdnustat} states that for a small weight parameter $\alpha$, the time-averaged mean-field GDA dynamics \eqref{flow-multi-type} converges to its invariant measure exponentially fast. This improves a prior algebraic convergence rate obtained by \cite{kim2023symmetric} for a similar time-averaged mean-field dynamics but with a polynomially decaying weight. On the other hand, we note that the  convergence rate $\lambda = O(e^{-M_0 \tau^{-1}})$ when the temperature $\tau$ is small, which is consistent with the existing rates of convergence obtained for mean-field Langevin dynamics (\cite{chizat2022mean,nitanda2022convex}) and mean-field GDA dynamics (\cite{lu2023two}). We also remark that our convergence result assumes a small  $\alpha$, which only enters  the exponential discounting factor used for  time-averaging of the velocity, but all players run their dynamics at a uniform speed. This is in contrast to \cite{lu2023two}, where the convergence result there relies on time-scale separation between different players.
\end{remark}

\subsection{Convergence of annealed dynamics}\label{sec:anneal}
In this section, we study the convergence of the following annealed time-averaged mean-field GDA dynamics
\begin{equation}
\label{flow-multi-type-annealed}
\forall i\in\cco 1,K\ccf,\quad \left\{
\begin{array}{rcl}
\partial_t \nu_t^i   &= &  \na_x \cdot \po \nu_t^i \na_x U_i\star \hat\nu_t^i \pf + \tau_t \Delta_x \nu_t^i   \\  
\partial_t \hat \nu_t^i &= & \alpha_t  \po \nu_t^i - \hat\nu_t^i \pf\,.
\end{array}\right.
\end{equation}
In the above, $\tau_t$ is a cooling function of time that decreases to zero as $t$ tends to infinity, and $\alpha_t>0$ is a time-dependent weight exponent.  

We aim to show that by choosing the cooling schedule $\tau_t$ and the weight exponent $\alpha_t$ appropriately, the solution $\{\nu_t^i\}$ converges to an MNE of the unregularized energies $\{\mE_i\}_{i=1}^K$.  To this end, we first track the evolution of certain entropies along the annealed dynamics in the next proposition, which extends  Proposition \ref{prop:Lyapunov-multi} by taking account of time-varying temperature $\tau_t$ and weight $\alpha_t$. To state the result, it is useful to set $\rho_{t,\tau}^i\propto \exp(-U_i\star\hat\nu_t^{-i}/\tau)$ for $i\in\cco 1,K\ccf$. We also write 
    \[s_{t,\tau}:= \sum_{i=1}^K \co   \mathcal H\po \nu_t^i|\rho_{t,\tau}^i \pf + \mathcal H\po \hat \nu_t^i|\rho_{t,\tau}^i \pf\cf \,.  \]

\begin{prop} For constants $\delta$ and $\beta$ such that 
\begin{equation} \label{eq:cond_annealing-multi}
    0<\delta \leqslant \frac1{12M_0}\,, 0< \beta \leqslant  \frac3{  \delta^3 M_2}\,,
\end{equation} we choose $c_0 >1$ large enough so that
    \begin{equation}
        \label{eq:cond_annealing-multi-c0}
c_0 \geqslant  \frac{ 144}{\beta^2}\,,\qquad \delta  M_2 c_0^{\delta M_0} \ln c_0  \geqslant 8 \pi^2 \,. 
        \end{equation}
    We set     
    \label{prop:annealing-multi}
    \[ \tau_t = \frac{1}{\delta \ln(c_0+t)}\,,\qquad \alpha_t = \frac{\beta}{\sqrt{c_0 + t}} \,,\]
        Then, 
    for all $t\geqslant 0$,
   \begin{equation}
\label{eq:stau}
       s_{t,\tau_t} \leqslant \frac{32\delta M_1  }{\beta \sqrt{c_0+t}} + e^{-\frac{\beta}{2}\po \sqrt{c_0 + t} - \sqrt{c_0}\pf }s_0 \,. 
   \end{equation}
\end{prop}
The proof of Proposition \ref{prop:annealing-multi} can be found in Section  \ref{sec:prooffor23}. 
As a corollary of Proposition \ref{prop:annealing-multi}, the next theorem characterizes the large-time  closedness of the annealed dynamics  to the proximal-Gibbs measures $\{\nu^{i,\ast}_{\tau_t}\}_{i=1}^K$. 

\begin{thm}\label{thm:proxygibbs} 
     Under the same setting as Proposition \ref{prop:annealing-multi}, there exists a constant $C>0$ depending only on $M_0, M_1, M_2$ and $s_0$ such that  for sufficiently large $t$,
     $$
     \|\nu^i_t - \nu^{i,\ast}_{\tau_t}\|_{TV}  \leq C(t+c_0)^{-1/4}, \ \forall i\in\cco 1,K\ccf.
     $$
\end{thm}
As a consequence of Theorem \ref{thm:proxygibbs}, the next theorem establishes the convergence of the annealed dynamics to the MNE of the unregularized energies with respect to the NI-error. 

\begin{thm} \label{thm:mainanneal}
    Under the same setting as Proposition \ref{prop:annealing-multi}, there exists a constant $C^\prime>0$ depending only on $M_0, M_1, M_2, L$, $s_0$ and $d_i,i\in\cco 1,K\ccf$ such that  for sufficiently large $t$,
    $$
     \NI(\nu_t) \leq \frac{C^\prime \log \log t}{\log t}.
    $$
\end{thm}

\begin{remark}
 The convergence rate $O(\log \log t/\log t)$  in Theorem \ref{thm:mainanneal} was previously obtained in \cite{lu2023two} for an annealed two-scale mean-field GDA dynamics. However, we emphasize again that the key ingredient of the proof --- Proposition \ref{prop:annealing-multi}, does not require scale-separation among the players.  
\end{remark}

\section{Proofs of Main Results}
In this section, we provide the proofs of the main theorems and the key propositions.

\subsection{Proof of the result in Section  \ref{sec:setting}}  \label{sec:proof1}

\begin{proof}[Proof of Proposition \ref{prop:NEmulti}]
Consider a solution of~\eqref{flow-multi-type} with finite initial entropy. Since $\partial_t \hat\nu_t=\alpha(\nu_t - \hat\nu_t) $,
\begin{multline*}
    \|\hat \nu_{t+s}-\hat\nu_t\|_{TV} \leqslant \alpha \int_t^{t+s} \|\hat \nu_u-\nu_u\|_{TV}\dd u \\
    \leqslant \sqrt{2} \alpha \int_t^{t+s} \sqrt{\|\hat \nu_u-\rho_u\|_{TV}^2 + \|\rho_u-\nu_u\|_{TV}^2}\dd u  \leqslant  \frac{4\sqrt{s_0}\alpha}{\lambda} e^{-\lambda t/2}\,.
\end{multline*}
By completeness, $\hat\nu_t$ then converge in total variation to  some $\nu^*_{\tau}$.  Since $\|\hat\nu_t - \rho_t\|_{TV}\leqslant \sqrt{2s_t}\rightarrow0$, we also get that $\rho_t$ converge to $\nu^*_{\tau}$ as $t\rightarrow \infty$.
Moreover, for all $i\in\cco 1,K\ccf$, 
\[\|U_i\star \hat\nu_{t}^{-i} - U_i\star \nu^{-i,*}_{\tau}\|_{\infty} \leqslant M_0 \| \|  \hat\nu_{t}^{-i} -  \nu^{-i,*}_{\tau}\|_{TV} \underset{t\rightarrow\infty}\longrightarrow{0}\,.\]
This shows that $\rho_t^i$ converges to the density proportional to $\exp(-U_i\star\nu^{-i,*}_{\tau})$. By uniqueness of the limit, we get that $\nu^*_{\tau}$ solves the fixed-point equation~\eqref{eq:fixedpointmulti}, hence is an MNE.

    Using this fixed point equation   and the pairwise zero-sum assumption that  $W_{ij}=-W_{ji}$, which in particular implies that for any $\nu\in\mathcal P(\mX)$,
\[\sum_{i=1}^K \sum_{j\neq i} \int_{\mX_i\times \mX_j} W_{ij}\nu^i \nu^j = \sum_{i=1}^K  \int_{\mX_i\times\mX_{-i}} U_i \nu^i   \nu^{-i}   = 0 \,, \]
we get, for any $\nu \in \mathcal P(\mX)$,
\begin{align*}
    \NI(\nu)     & \geqslant \sum_{i=1}^K \co \mE_{i,\tau}\po\nu^i,\nu^{- i}\pf  -   \mE_{i,\tau}  \po\nu^{i,\ast}_{\tau},\nu^{- i}\pf  \cf \\
    &= \sum_{i=1}^K \co \tau \int_{\mX_i} \nu^i \ln \nu^i - \tau \int_{\mX_i} \nu^{i,\ast}_{\tau} \ln \nu^{i,\ast}_{\tau}  - \sum_{j\neq i} \int_{\mX_i\times \mX_j} W_{ij}\nu^{i,\ast}_{\tau}   \nu^j \cf  \\
    &= \sum_{i=1}^K \co \tau \int_{\mX_i} \nu^i \ln \nu^i - \tau \int_{\mX_i} \nu^{i,\ast}_{\tau} \ln \nu^{i,\ast}_{\tau}  - \sum_{j\neq i} \int_{\mX_i\times \mX_j} W_{ij}\nu^{i,\ast}_{\tau}   (\nu^j-\nu_*^j) \cf  \\
    &= \sum_{i=1}^K \co \tau \int_{\mX_i} \nu^i \ln \nu^i - \tau \int_{\mX_i} \nu^{i,\ast}_{\tau} \ln \nu^{i,\ast}_{\tau}  \cf +  \sum_{j=1}^K \int_{\mX_i\times \mX_j} \po \sum_{i\neq j} W_{ji}\nu^{i,\ast}_{\tau} \pf   (\nu^j-\nu_*^j)   \\
    &= \tau \sum_{i=1}^K \co   \int_{\mX_i} \nu^i \ln \nu^i -   \int_{\mX_i} \nu^{i,\ast}_{\tau} \ln \nu^{i,\ast}_{\tau}  - \int_{\mX_i} \ln \nu^{i,\ast}_{\tau} (\nu^i - \nu^{i,\ast}_{\tau}) \cf \\
    &= \tau \sum_{i=1}^K \mathcal H \po \nu^i|\nu^{i,\ast}_{\tau}\pf \,.  
\end{align*}
This implies the uniqueness of the mixed Nash equilibrium. Indeed, if $(\nu^1,\dots,\nu^K)$ is a Nash equilibrium, then the left-hand side of \eqref{eq:entropy_ineq} is zero, implying that $\nu^i= \nu^{i,\ast}_{\tau}$ for all $i\in\cco 1,K\ccf$.
\end{proof}

\subsection{Proofs of results in Section \ref{sec:fixedtemp}} \label{sec:prooffor22}

\begin{proof}[Proof of Proposition \ref{prop:Lyapunov-multi}]
Since $\rho_t^i$ is the invariant measure of the Markov generator $\na  U_i\star \hat\nu_t^{-i} \cdot \na +  \tau \Delta$, by usual computations on the entropy dissipation along the corresponding flow,
\begin{eqnarray*}
\partial_t \mathcal H\po \nu_t^i | \rho_t^i \pf &= & (\partial_s)_{|s=t} \mathcal H\po \nu_s^i | \rho_t^i\pf + (\partial_s)_{|s=t} \mathcal H\po \nu_t^i | \rho_s^i \pf \\
& = & -  \tau \mathcal I\po \nu_t^i  | \rho_t^i\pf-\int_{\mX_i} \nu_t^i \partial_t \ln \rho_t^i \\
& = & -  \tau \mathcal I\po \nu_t^i  | \rho_t^i \pf + \frac1 \tau \int_{\mX_i} \partial_t \po \sum_{j\neq i} W_{ij}\star \hat\nu_t^{j} \pf \po \rho_t^i - \nu_t^i\pf   \\
& = & -  \tau \mathcal I\po \nu_t^i  | \rho_t^i\pf + \frac \alpha \tau \sum_{j\neq i} \int_{\mX_i} W_{ij} \star (\nu_t^{j} - \hat \nu_t^{j})  \po \rho^i_t  - \nu_t^i\pf   \\
& \leqslant &  -  \tau \mathcal I\po \nu_t^i  | \rho_t^i \pf + \frac{\alpha} \tau  \sum_{j\neq i}  \|\na_{x_i}  W_{ij}\star(\nu_t^{j} - \hat \nu_t^{j})\|_{\infty} \mathcal W_1 \po  \rho^i_t ,\nu_t^i\pf   \,.
\end{eqnarray*}
Using that $\|U_i\star \nu^{-i}\|_\infty \leqslant M_0$, reyling on the log-Sobolev inequality for the Lebesgue measure on $\T^d$ (with constant $1/(8\pi^2)$, see \cite[Proposition 5.7.5(ii)]{BakryGentilLedoux}) and the Holley-Stroock perturbation result (e.g. \cite[Proposition 5.1.6]{BakryGentilLedoux}), we see that each $\rho_t^i$  (hence $\rho_t$ by tensorization)  satisfies a log-Sobolev inequality with constant $ \kappa $ given by \eqref{eq:kappa}, meaning that
\begin{equation*}
    \forall \mu \in\mathcal P(\mX)\,,\qquad \mathcal H(\mu|\rho_t) \leqslant \kappa \mathcal I(\nu|\rho_t)\,.
\end{equation*}
 From \cite{Otto2000}, it implies the Talagrand inequality
\begin{equation*}
\forall \mu \in\mathcal P(\mX)\,,\qquad \mathcal W_2^2(\mu,\rho_t) \leqslant 4\kappa \mathcal H(\nu|\rho_t)\,.
\end{equation*}
At that point, if we were following \cite{kim2023symmetric}, we would have a learning rate $\alpha_t$ depending on time and, from the previous bounds using the log-Sobolev and Talagrand inequalities for $\rho_t$, that $\mathcal W_1\leqslant\mathcal W_2$, we would deduce that 
\begin{eqnarray*}
\partial_t \mathcal H\po \nu_t^i | \rho_t^i \pf
& \leqslant &  - \frac{\tau }{\kappa  } \mathcal H\po \nu_t^i  | \rho_t^i \pf + 4  \frac{\alpha_t}{\tau} \sum_{j\neq i} \|\na_{x_i} W_{ij} \|_{\infty} \sqrt{\kappa \mathcal H \po \nu_t^i | \rho^i_t\pf}   \,,
\end{eqnarray*}
from which we get that $ \mathcal H\po \nu_t^i |\rho_t^i\pf$ vanishes as $t\rightarrow \infty$ provided this is the case of $\alpha_t$ -- which prevents exponential convergence so we want to avoid this condition.

Rather, we bound 
$$\begin{aligned}
    \|\na_{x_i} W_{ij}\star( \nu_t^{j} - \hat\nu_t^{j})\|_{\infty}  & 
 \leqslant   \|\na_{x_ix_j}^2 W_{ij}\|_\infty \mathcal W_1(\nu_t^{j} ,\hat\nu_t^{j}) 
 \\
 & \leqslant \|\na_{x_ix_j}^2 W_{ij}\|_\infty \co \mathcal W_1(\nu_t^{j} ,\rho_t^j) + \mathcal W_1(\rho_t^j,\hat\nu_t^{j})\cf .
\end{aligned}
$$
Plugging this in the previous inequality and then using the log-Sobolev and Talagrand inequalities for $\rho_t$ yields
\begin{align}
    \partial_t \mathcal H\po \nu_t | \rho_t \pf 
& \leqslant  - \tau \mathcal I\po \nu_t  | \rho_t \pf + \frac{\alpha}{\tau} \sum_{i=1}^K \sum_{j\neq k} \Big[\|\na^2_{x_ix_{j}} W_{ij}\|_\infty \mathcal W_1 \po \rho^i_t  , \nu_t^i\pf \nonumber\\
&\hspace{4cm} \times  \po   \mathcal W_1\po \nu_t^{j}  , \rho^{j}_t  \pf   +  \mathcal W_1\po \rho^{j}_t, \hat\nu_t^{j}  \pf  \pf   \Big]\nonumber \\
& \leqslant   - \frac{\tau}{\kappa } \mathcal H\po \nu_t  | \rho_t \pf + 4\kappa \frac{\alpha}\tau \sum_{i=1}^K \sum_{j\neq k} \Bigg[  \|\na^2_{x_ix_{j}} W_{ij}\|_\infty \sqrt{\mathcal H \po  \nu_t^i |\rho^i_t \pf}  \nonumber \\
&\hspace{4cm} \times   \po   \sqrt{\mathcal H\po \nu^{j}_t | \rho^{j}_t   \pf}   + \sqrt{ \mathcal H\po \hat\nu_t^{j} |\rho^{j}_t    \pf}  \pf\Bigg]  \nonumber  \\
& \leqslant   \po (2+\theta) \alpha M' - \frac{\tau}{\kappa }\pf  \mathcal H\po \nu_t  | \rho_t\pf +\frac{\alpha M' }{\theta} \mathcal H\po \hat\nu_t |\rho_t    \pf  \,, 
 \label{eq:rtRt}
\end{align}
with $M' = 2\kappa  M_2/\tau $, for any $\theta>0$. We have also used the tensorization property of the relative entropy.

Now, using that $\hat\nu_t^i$ is a probability measure for all $t\geqslant 0$,
\begin{eqnarray}
\partial_t \mathcal H\po \hat\nu_t^i | \rho_t^i \pf &= & \int_{\mX} \ln \frac{\hat\nu_t^i}{\rho_t^i} \partial_t \hat\nu_t^i - \int_{\mX_i} \partial_t \po \ln \rho_t^i \pf \hat\nu_t^i \nonumber\\
 &= & \alpha \int_{\mX} \ln \frac{\hat\nu_t^i}{\rho_t^i}\po \nu_t^i - \hat\nu_t^i\pf  + \frac1\tau  \int_{\mX_i} \partial_t \po \sum_{j\neq i} W_{ij} \star \hat\nu_t^{j}\pf  \po  \hat\nu_t^i - \rho_t^i \pf\nonumber  \\ 
  &= & \alpha \int_{\mX} \ln \frac{\hat\nu_t^i}{\rho_t^i}\nu_t^i   - \alpha \mathcal H\po \hat\nu_t^i | \rho_t^i \pf   + \frac \alpha \tau \sum_{j\neq i}  \int_{\mX_i} W_{ij}\star (\nu_t^{j} -  \hat\nu_t^{j})  \po  \hat\nu_t^i - \rho_t^i \pf  \,. \label{uneeq}
\end{eqnarray}
For the first term, we use that
\[ \int_{\T^d} \ln \frac{\hat\nu_t^i}{\rho_t^i } \nu_t^i = \mathcal H \po \nu_t^i | \rho_t^i\pf - \mathcal H \po \nu_t^i | \hat\nu_t^i\pf \leqslant  \mathcal H \po \nu_t^i | \rho_t^i \pf\,. \] 
To treat the last term, for $i\neq j$, using that $W_{ij}=-W_{ji}$, write
\begin{eqnarray*}
\lefteqn{ 
\int_{\mX_i\times\mX_j} W_{ij}  \po  \hat\nu_t^i - \rho_{t}^i\pf  (\nu_t^j -  \hat \nu_t^j)    + \int_{\mX_i\times\mX_j} W_{ji}  \po  \hat\nu_t^j - \rho_{t}^j\pf  (\nu_t^i -  \hat\nu_t^i)
}\\
&= & \int_{\mX_i\times\mX_j}  W_{ij}  \po  \nu_t^i - \rho_{t}^i\pf (\nu_t^j -  \hat \nu_t^j)    +  \int_{\mX_i\times\mX_j}  W_{ji}  \po  \nu_t^j - \rho_{t}^j\pf (\nu_t^i -  \hat\nu_t^i) \\
& \leqslant & \|\na_{x_ix_j}^2 W_{ij}\|_\infty \co \mathcal W_1\po \nu_t^j,\hat \nu_t^j\pf \mathcal W_1\po \nu_t^i,\rho_{t}^i \pf +  \mathcal W_1\po \nu_t^j,\rho_{t}^j\pf \mathcal W_1\po \nu_t^i,\hat\nu_t^i\pf   \cf  \,.
\end{eqnarray*}
Summing~\eqref{uneeq} over $i\in\cco 1,K\ccf$ and using the last two estimates, one has that 
\[
    \partial_t \mathcal H\po \hat\nu_t | \rho_t \pf  \leqslant   \alpha \mathcal H\po \nu_t | \rho_t \pf - \alpha \mathcal H\po \hat\nu_t | \rho_t \pf + \frac{2\alpha}{\tau}\sum_{i=1}^K\sum_{j\neq i} \|\na_{x_ix_j}^2 W_{ij}\|_\infty  \mathcal W_1\po \nu_t^j,\hat \nu_t^j\pf \mathcal W_1\po \nu_t^i,\rho_{t}^i \pf\,.
\]
For distinct $i$ and $j$, we bound, for any $\theta'>0$,
\begin{eqnarray*}
2\mathcal W_1\po \nu_t^j,\hat \nu_t^j\pf \mathcal W_1\po \nu_t^i,\rho_{t}^i \pf & \leqslant & 2\po \mathcal W_1\po \nu_t^j,\rho_{t}^j \pf  + \mathcal W_1\po \rho_{t}^j,\hat \nu_t^j\pf \pf \mathcal W_1\po \nu_t^i,\rho_{t}^i \pf \nonumber \\
& \leqslant &   \mathcal W_1^2 \po \nu_t^j,\rho_{t}^j \pf + \frac{1}{\theta'} \mathcal W_1^2\po \rho_{t}^j,\hat \nu_t^j\pf + (1+\theta') \mathcal W_1^2\po \nu_t^i,\rho_{t}^i \pf\,.
\end{eqnarray*}
Plugging this in the previous inequality and then using the Talagrand inequalities for $\rho_{t}^i$ leads to
\[\partial_t  \mathcal H\po \hat\nu_t | \rho_t \pf \leqslant - \alpha \po 1 - \frac{2\kappa   M_2}{\tau \theta'} \pf  \mathcal H\po \hat\nu_t | \rho_t \pf  +  \alpha \po 1+ \frac{2\kappa   M_2      (2+ \theta')}{\tau} \pf  \mathcal H\po \nu_t | \rho_t \pf \,.  \] 
Taking $\theta' = 4\kappa  M_2 /\tau $, writing $M''= 1+ 2\kappa   M_2       (2+ \theta')/\tau $ and summing this inequality with \eqref{eq:rtRt}  we get, for any $\theta>0$, 
\[ 
\partial_t s_t \leqslant   \po (2+\theta) \alpha M' + \alpha M'' - \frac{\tau }{\kappa }\pf \mathcal H\po \nu_t | \rho_t \pf  +     \alpha\po \frac{ M'}\theta - \frac12\pf  \mathcal H\po \hat\nu_t | \rho_t \pf   \,.
\]
Taking $\theta = 4M'$  and assuming that
\[\alpha \leqslant \frac{\tau}{2\kappa \po (2+\theta)   M' + M'' \pf } =  \frac{\tau}{4\kappa \po (2+8\kappa M_2/\tau )   \kappa M_2/\tau  +  1+ 2\kappa   M_2       (2+ 4\kappa M_2/\tau)/\tau \pf }\,,\]
which is implied by $\alpha \leqslant \alpha_0$ with $\alpha_0$ given by \eqref{eq:kappa}, we get
\begin{equation}
    \label{eq:rhatr}
    \partial_t s_t   \leqslant -  \frac{\tau }{2\kappa }  \mathcal H\po \nu_t | \rho_t \pf   - \frac{\alpha}4    \mathcal H\po \hat \nu_t | \rho_t \pf  \leqslant - \lambda s_t\,, 
\end{equation}
which concludes the proof of the proposition. 
\end{proof}

\begin{proof}[Proof of Theorem \ref{thm: bdnustat}] 
First, it follows from Proposition \ref{prop:NEmulti} that 
$$
\begin{aligned}
    \sum_{i=1}^K   \mathcal H \po \hat \nu_t^i|\nu^{i,\ast}_\tau \pf   \leq \tau^{-1} \NI_{\tau}(\hat\nu_t) = \tau \mH(\hat\nu_t | \rho_t),
\end{aligned}
$$
where we have used \eqref{eq:NIhat} in the last identity. Therefore the first estimate of the theorem follows directly from above and Proposition \ref{prop:Lyapunov-multi}. For the second one, by the triangular and Pinsker inequality, 
\begin{align*}
    \sum_{i=1}^K \|\nu_t^i - \nu^{i,\ast}_\tau\|_{TV}^2 & \leqslant \sum_{i=1}^K  \po \|\nu_t^i - \rho_{t}^i \|_{TV} + \| \rho_{t}^i -  \hat \nu_t^i \|_{TV} + \|\hat \nu_t^i - \nu^{i,\ast}_\tau\|_{TV}  \pf^2 \\
    &\leqslant 6 \sum_{i=1}^K\co  \mathcal H\po   \nu_t^i |   \rho_{t}^i  \pf     +  \mathcal H\po   \hat \nu_t^i |   \rho_{t}^i  \pf     +  \mathcal H\po  \hat \nu_t^i |   \nu^{i,\ast}_\tau  \pf  \cf  \\
    &\leqslant 12 s_t \,. 
\end{align*}
Therefore the second estimate follows again from Proposition~\ref{prop:Lyapunov-multi} and the first estimate.
\end{proof}

\subsection{Proofs for results in Section \ref{sec:anneal}} \label{sec:prooffor23}

\begin{proof}[Proof of Proposition \ref{prop:annealing-multi}]
We decompose 
\[\partial_t (s_{t,\tau_t}) = (\partial_u)_{|u=t} s_{u,\tau_t} + (\partial_u)_{|u=\tau_t} s_{t,u} \,. \]
The second term is 
\begin{equation}
     \label{eq:annealed1}
     (\partial_u)_{|u=\tau_t} s_{t,u} =  - \frac{\partial_t  \tau_t}{\tau_t^2}     \sum_{i=1}^K \int_{\T^d}   U_i\star \hat\nu_t^{-i}  \po 2 \rho_{t}^i - \nu_t^i - \hat\nu_t^i\pf  \leqslant \frac{4\delta M_1}{c_0+t }     \,. 
\end{equation}
The first term is treating as in  the proof of Proposition~\ref{prop:Lyapunov-multi}, where \eqref{eq:rhatr} becomes
\begin{equation}
     \label{eq:annealed2} 
     (\partial_u)_{|u=t} s_{u,\tau_t}  \leqslant - \lambda_t s_{t,\tau_t}
 \end{equation}
 as soon as $\alpha_t \leqslant \bar\alpha_{0,t}$ with
\begin{equation*}
\kappa_t =\frac{\exp\po  \tau_t^{-1} M_0\pf  }{8 \pi^2 }\,,\qquad \lambda_t = \min\po \frac{\tau_t}{2\kappa_t},\frac{\alpha_t}{4}\pf \,,\qquad \bar \alpha_{0,t}= \frac{\tau_t}{53 \kappa_t \max(1,(\kappa_t M_2/\tau_t)^2) }
 \,.
 \end{equation*} 
 Let us check that the conditions~\eqref{eq:cond_annealing-multi} and \eqref{eq:cond_annealing-multi-c0} ensure that $\alpha_t \leqslant \bar\alpha_{0,t}$ for all $t\geqslant 0$. First,
\[\frac{\kappa_t M_2 }{\tau_t } =\frac{\delta  M_2 }{8 \pi^2 } (c_0+t)^{\delta M_0} \ln(c_0+t) \geqslant 1 \,, \]
so that
\[ \frac{\alpha_t}{\bar \alpha_{0,t}}= \frac{53 \alpha_t \kappa_t^3   M_2^2 }{\tau_t^3 } = \frac{53 \delta^3 \beta M_2^2 }{(8 \pi^2)^3 } (c_0+t)^{3\delta M_0-1/2} \ln^3(c_0+t)\,.  \]
Using \eqref{eq:cond_annealing-multi} and checking (e.g. numerically) that $x^{-1/4}\ln^3(x) \leqslant 90$ for all $x\geqslant 1$, we get that 
\[ \frac{\alpha_t}{\bar \alpha_{0,t}} \leqslant \frac{\delta^3 \beta M_2^2 }{3 } \leqslant  1 \,,  \]
as required. Moreover, $\alpha_t \leqslant \bar\alpha_{0,t} \leqslant \tau_t / (53\kappa_t)$, which means that $\lambda_t = \alpha_t/4$. 

Gathering~\eqref{eq:annealed1} and \eqref{eq:annealed2}, we have thus obtained 
 \[\partial_t s_t \leqslant - \frac{\beta}{4\sqrt{c_0+t}}s_t + \frac{4\delta M_1 }{c_0+t}\,.\]
 For any $C>0$,
\begin{eqnarray*}
\partial_t \po s_t  - \frac{C}{\sqrt{c_0+t}}\pf  &\leqslant& - \frac{\beta}{4\sqrt{c_0+t}}s_t  + \frac{4\delta M_1}{c_0+t} + \frac{3C}{2(c_0+t)^{3/2}}  \\
& \leqslant & - \frac{\alpha_t}{4} \po s_t - \frac{C}{\sqrt{c_0+t}}\pf
\end{eqnarray*}
provided 
\[ \frac{4\delta M_1 }{c_0+t} + \frac{3C}{2(c_0+t)^{3/2}} \leqslant \frac{C\beta}{4(c_0+t)}\,,\]
which holds true thanks to the first part of~\eqref{eq:cond_annealing-multi-c0} and by taking
\[C = \frac{32\delta M_1 }{\beta}\,.\]
Then, we conclude by
\[s_t  \leqslant \frac{C}{\sqrt{c_0+t}}  + e^{-\int_0^t \frac{\alpha_u}{4}\dd u }s_0  =  \frac{C}{\sqrt{c_0+t}} + e^{-\frac{\beta}{2}\po \sqrt{c_0 + t} - \sqrt{c_0}\pf }s_0 \]
\end{proof}

\begin{proof}[Proof of Theorem~\ref{thm:proxygibbs}]
    The argument is the same as in the proof of Theorem~\ref{thm: bdnustat}: using Proposition~\ref{prop:NEmulti} with the triangular and Pinsker inequalities,
    \begin{align*}
        \sum_{i=1}^K \|\nu^i_t - \nu^{i,\ast}_{\tau_t}\|_{TV}^2 &\leqslant 3\sum_{i=1}^K \co \|\nu^i_t - \rho_{t,\tau_t}^i \|_{TV}^2 + \|\rho^i_{t,\tau_t} - \hat\nu^i_t \|_{TV}^2 + \|\hat \nu^i_t - \nu^{i,\ast}_{\tau_t}\|_{TV} ^2 \cf \\
        & \leqslant 6 \sum_{i=1}^K \co \mathcal H\po \nu^i_t | \rho_{t,\tau_t}^i\pf + \mathcal H\po \hat \nu^i_t | \rho_{t,\tau_t}^i\pf+  \mathcal H\po \hat \nu^i_t | \nu^{i,\ast}_{\tau_t}\pf \cf \\
        &\leqslant 12 s_{t,\tau_t}\,.
    \end{align*}
    The estimate of the theorem follows from Proposition~\ref{prop:annealing-multi}, Proposition \ref{prop:NEmulti} and the fact that the exponential term on the right side of \eqref{eq:stau} is dominated by the first term for large $t$.
\end{proof}

Next, we turn to the proof of main  Theorem \ref{thm:mainanneal} which establishes the convergence of the annealed mean-field dynamics under the NI-error. To this end, recall that $\{\nu^{i,\ast}_{\tau}\}_{i=1}^K$ is the MNE associated to the entropy-regularized energies  $
\{\mE_{i,\tau}\}_{i=1}^K$. Recall also the Nikaid\^o-Isoda error associated to the unregularized energies $\{\mE_i\}_{i=1}^K$:
$$
\NI(\{ \nu^{i,\ast}_{\tau}\}_{i=1}^K) = \sum_{i=1}^K \mE_i(\nu^{i,\ast}_{\tau}, \nu^{- i,\ast}_{\tau}) - \inf_{\nu^i} \mE_i(\nu^{i}, \nu^{- i,\ast}_{\tau}) 
.$$ 
For a fixed $\epsilon>0$, we say that $\nu = (\nu^i)_{i=1}^K$ is an $\epsilon$-MNE of  $\{\mE_i\}_{i=1}^K$ if 
$$
\NI(\nu) \leq \epsilon.
$$
The following lemma is useful to bound $ \NI(\{ \nu^{i,\ast}_{\tau}\}_{i=1}^K)$, which generalizes \cite[Theorem 5]{domingo2020mean} and \cite[Lemma A.4]{lu2023two} for two-player games to the multi-player setting. 

\begin{lemma}\label{lem:laplacem} Under Assumption \ref{ass-multi}, there exists a constant $C>0$ depending only on $M_0, L, d_{i},i=1,\cdots,K$ such that $\{ \nu^{i,\ast}_{\tau}\}_{i=1}^K$ is an $\epsilon$-MNE of $\{\mE_i\}_{i=1}^K$ with $\epsilon>0$ if 
$$
\tau \leq \frac{C\epsilon}{\log \epsilon^{-1}}
$$
Alternatively, if $\tau$ is sufficiently small, then $\{ \nu^{i,\ast}_{\tau}\}_{i=1}^K$ is an $\epsilon$-MNE with 
\begin{equation}\label{eq:epsNashm}
  \epsilon = \beta \tau(\log (1/\tau)) \text{ for } \beta > \max_i d_{\mX_i} +1.
\end{equation}

\end{lemma}

\begin{proof}
    The proof essentially follows from that of \cite[Theorem 5]{domingo2020mean}. We only highlight the key idea below. Define
    $$
    V_i(x_i) = \sum_{j\neq i} W_{ij} \star  \nu^{j,\ast}_\tau, i=1,\cdots, K.
    $$
    Then we can rewrite the NI error $\NI(\{ \nu^{i,\ast}_{\tau}\}_{i=1}^K) $ as 
    $$
    \NI(\{ \nu^{i,\ast}_{\tau}\}_{i=1}^K) = \sum_{i=1}^K \Big(\int V_i(x_i) \nu_{\tau}^{i,\ast}(dx_i) - \min_{x_i^\prime} V_i(x_i^\prime)\Big).
    $$
    Using  that $\{ \nu^{i,\ast}_{\tau}\}_{i=1}^K$ satisfies the fixed point equation 
    $$
    \nu^{i,\ast}_{\tau} \propto e^{-V_i(x)/\tau}, 
    $$
    one can follow the same arguments from the proof of \cite[Theorem 5]{domingo2020mean} to obtain that $\{ \nu^{i,\ast}_{\tau}\}_{i=1}^K$  is an $\epsilon$-Nash equilibrium provided that for any $i\in 1,\cdots, K$, 
$$
e^{\frac{\eps}{2\tau}}  \frac{\text{Vol}(B_{\mX_i}^\delta)}{1- \text{Vol}(B_{\mX_i}^\delta)}  \geq 2 (M_0/\epsilon - 1),
$$
where $B_{\mX_i}^\delta$ is the centered ball of radius $\delta$ on $\mX_i$ with $\delta = \epsilon/(2L)$. Thanks to the lower bound on the volume of small balls in $\mX_i = \mathbb{T}^{d_i}$, the inequality above holds if 
$$
e^{\frac{\epsilon}{2\tau}} \geq C_1  \epsilon^{-(\max_i d_{i} +1)}
$$
for some $C_1>0$ depending only on $M_0, L,d_{i}$. This proves that the choice $\tau \leq C\epsilon/\log \epsilon^{-1}$ is sufficient to guarantee that $\{ \nu^{i,\ast}_{\tau}\}_{i=1}^K$ is an $\epsilon$-MNE of $\{\mE_i\}_{i=1}^K$. 
Moreover, with the choice \eqref{eq:epsNashm} for $\epsilon$ with $\beta > \max_i d_{i} +1$, one has for $\tau$ sufficiently small that 
$$
\begin{aligned}
    e^{\frac{\eps}{2\tau}} \geq  \frac{1}{\tau^{\beta}} > \frac{C_1}{\beta^{\max_i d_{i}+ 1} }  \Big(\frac{1}{\tau \log(1/\tau)}\Big)^{\max_i d_{i} + 1} = C_1 \epsilon^{-(\max_i d_{i} +1)}.
\end{aligned}
$$

\end{proof}

With the help of Theorem \ref{thm:proxygibbs}  and Lemma \ref{lem:laplacem}, we are ready to present the proof of Theorem \ref{thm:mainanneal}. 

\begin{proof}[Proof of Theorem \ref{thm:mainanneal}] 
  First of all, note that 
    $$
    \begin{aligned}
         & \NI(\nu_t) -  \NI(\nu_{\tau_t}^\ast) \\
         & = \sum_{i=1}^K \Big( \mE_i(\nu^{i}_{t}, \nu^{-i}_{t}) - \inf_{\nu^i} \mE_i(\nu^{i}, \nu^{- i}_{t}) \Big) - \sum_{i=1}^K \Big( \mE_i(\nu^{i,\ast}_{\tau_t}, \nu^{- i,\ast}_{\tau_t}) - \inf_{\nu^i} \mE_i(\nu^{i}, \nu^{- i,\ast}_{\tau_t}) \Big)\\
         & = \sum_{i=1}^K \Big(\mE_i(\nu^{i}_{t}, \nu^{- i}_{t}) - \mE_i(\nu^{i,\ast}_{\tau_t}, \nu^{- i,\ast}_{\tau_t})\Big) + \sum_{i=1}^K \Big(\inf_{\nu^i} \mE_i(\nu^{i}, \nu^{- i,\ast}_{\tau_t}) - \inf_{\nu^i} \mE_i(\nu^{i}, \nu^{- i}_{t})\Big) \\
         & =: I_1 + I_2
    \end{aligned}
    $$
    To bound $I_1$ for the difference of the NI-errors, note that for any $i\in 1,\cdots,K$, 
    $$
\begin{aligned}
    \mE_i(\nu^{i}_{t}, \nu^{- i}_{t}) - \mE_i(\nu^{i,\ast}_{\tau_t}, \nu^{- i,\ast}_{\tau_t})&
    = \mE_i(\nu^{i}_{t}, \nu^{- i}_{t}) - \mE_i(\nu^{i,\ast}_{\tau_t}, \nu^{- i}_{t}) +  \mE_i(\nu^{i,\ast}_{\tau_t}, \nu^{- i}_{t}) - \mE_i(\nu^{i,\ast}_{\tau_t}, \nu^{- i,\ast}_{\tau_t}) \\
    & \leq 2M_0 \|\nu^{i}_{t} - \nu^{i,\ast}_{\tau_t}\|_{TV}.
 \end{aligned}
    $$
    Similarly,  for $I_2$ we have that for any $i$,  
       $$
\begin{aligned}
\inf_{\nu^i} \mE_i(\nu^{i}, \nu^{- i,\ast}_{\tau_t}) - \inf_{\nu^i} \mE_i(\nu^{i}, \nu^{- i}_{t}) & \leq  \min_{x_i} \sum_{j\neq i} [W_{ij} \star (\nu^i_{t} - \nu^{i,\ast}_{\tau_t}) ](x_i) \\
& \leq M_0\|\nu^{i}_{t} - \nu^{i,\ast}_{\tau_t}\|_{TV}
     \end{aligned}
    $$
Combining the  estimates above, one has from Theorem \ref{thm:proxygibbs}  that there exists $C_1>0$ such that for large $t$, 
$$
\NI(\nu_t) -  \NI(\nu_{\tau_t}^\ast)  \leq C_1 (t+c_0)^{-1/4}.
$$
Furthermore, thanks to Lemma \ref{lem:laplacem} with the same choice of $\tau_t$, one has that for large $t$,
$$
 \NI(\nu_{\tau_t}^\ast) \leq \frac{C_2 \log\log t}{\log t},
$$
where $C_2>0$ depends on $M_0,L$ and $d_{i}, i\in\cco 1,K\ccf$. The theorem then follows directly from the last two estimates.

\end{proof}

\section{Conclusion and future work}
In this paper, we consider the approximation of the MNE of a multi-player zero-sum game with a time-averaged mean-field gradient flow dynamics.  We proved an exponential convergence in total variation metric for the constructed dynamics for finding the MNE of an entropy-regularized loss on the space of probability measures. We also proved that an annealed mean-field dynamics with a suitable cooling schedule converges to an MNE of the original unregularized game loss with respect to the Nikaid\` o-Isoda error.

We would like to highlight several interesting open directions for future research. First, our convergence analysis has been concentrated on a game objective defined by sum of pairwise zero-sum losses. It remains an interesting theoretical question to extend the  analysis to a general zero-sum loss that involves simultaneous interactions between multiple players. Next, although we have focused on the analysis on the tori, we anticipate that our proof strategies can be adapted to deal with game objectives  (with suitable confining condition) on unbounded state spaces. Furthermore, from a practical point of view, the mean field dynamics are often implemented by a certain interacting particle algorithm with a large number of particles. To perform a complete error analysis for the particle algorithm, it is crucial to establish a uniform-in-time quantitative error bound between the particle system and the mean field dynamics, which remains open. Lastly, our approach can potentially be extended to the general problem of sampling probability measures solving a fixed-point problem of the form
\[\nu \propto \exp (-F(\nu))\]
where $F$ is not the flat derivative of some energy. In that case, the  mean-field evolution naturally associated to this equation is not the Wasserstein gradient flow of an explicit free energy, so that standard mean-field entropy methods do not apply. However, with a time-averaged flow, the Lyapunov condition is explicitly given by relative entropies with respect to the  proximal-Gibbs measure.  We will investigate these questions in future works.

\section*{Acknowledgments}
YL thanks the support from the U.S. National Science Foundation  through the awards DMS-2343135 and
DMS-2436333. 
The research of PM is supported by the project  CONVIVIALITY (ANR-23-CE40-0003) of the French National Research Agency.

\bibliographystyle{siamplain}
\bibliography{references}

\newpage

\end{document}